\documentclass[a4paper,11pt]{article}
\usepackage[margin=1in]{geometry}
\usepackage{setspace,graphicx,url,hyperref}
\usepackage{amsmath,amssymb,amsfonts,amsthm,natbib}

\newcommand{\N}{\ensuremath{\mathbb{N}}}

\newcommand{\PP}{\ensuremath{\mathbb{P}}}
\newcommand{\E}{\ensuremath{\mathbb{E}}}
\newtheorem{lemma}{Lemma}
\newtheorem{condition}{Condition}
\newtheorem{theorem}{Theorem}

\title{On the expected runtime of multiple testing algorithms with bounded error}
\author{Georg Hahn}
\date{}

\begin{document}
\maketitle

\begin{abstract}
	Consider testing multiple hypotheses in the setting where the p-values of all hypotheses are unknown and thus have to be approximated using Monte Carlo simulations. One class of algorithms published in the literature for this scenario provides guarantees on the correctness of their testing result through the computation of confidence statements on all approximated p-values. This article focuses on the expected runtime of such algorithms and derives a variety of finite and infinite expected runtime results.
\end{abstract}

\textit{Keywords:}
algorithm; bounded error; computational effort; finite expected runtime; multiple testing.

\section{Introduction}
\label{sec:Introduction}
Consider the testing of $m \in \N$ hypotheses $H_{01},\ldots,H_{0m}$ in a scenario in which the p-values $p_1,\ldots,p_m$ corresponding to the $m$ hypotheses are unknown and thus have to be approximated using Monte Carlo simulations, for instance through bootstrap or permutation tests. Several algorithms published in the literature are designed for this scenario, either using a truncation rule to reach fast decisions \citep{BesagClifford1991, DavidsonMacKinnon2000, AndrewsBuchinsky2000, AndrewsBuchinsky2003, Wieringen2008, Sandve2011, SilvaAssuncao2013, SilvaAssuncao2018} or using a heuristic approach to minimize the computational effort without truncation \citep{Lin2005, Silva2009, GandyHahn2017}.

In the aforementioned Monte Carlo scenario, the main focus of this article lies on the expected runtime of methods which aim to provide a guarantee of correctness on their decision (rejection or non-rejection) for each hypothesis through the computation of a sequence of confidence intervals on each p-value. For the special case of a single hypothesis, it is known that algorithms which compute (or rely on) the decision of one hypothesis with respect to a fixed threshold have an infinite expected runtime. This result is restated below in Section~\ref{sec:single}.

The novelty of this article consists in the generalisation of expected runtime results from single to multiple hypotheses. To be precise, under a simple and weak asymptotic condition on the length of the intervals produced by the confidence sequence, the article shows the following three main results. For applications relying on independent testing, meaning decisions of multiple hypotheses tested at a constant (Bonferroni-type) threshold, all but two hypotheses can be decided in finite expected runtime (Section~\ref{sec:finite}). This result does not extend to applications which require full knowledge of all individual decisions, for instance step-up or step-down procedures, in which case no algorithm can guarantee even a single decision in finite expected runtime (Section~\ref{sec:multiple}). Simulations included in the conclusions (Section~\ref{sec:conclusions}) show that in practice, however, the number of pending decisions is typically low.

Although unconsidered in their original publications, the expected runtimes derived in this article apply to a whole class of published algorithms in the literature. For instance, they apply to the algorithms of \cite{GuoPedadda2008, GandyHahn2014} which provide a guarantee on the correctness of the decision on each hypothesis through the computation of \cite{ClopperPearson1934} confidence intervals. Expected runtimes also hold true for the confidence sequences of \cite{Robbins1970, Lai1976} employed in the methods of \cite{GandyHahn2016, Ding2018}, as they do for the confidence sequences of \cite{DarlingRobbins1967b, DarlingRobbins1967a} and the binomial confidence intervals of \cite{Armitage1958} employed in \cite{Fay2007,Gandy2009}. The results of this article do not apply to the \textit{push-out design} of \cite{FayFollmann2002} and the \textit{B-value design} of \cite{Kim2010}, which both achieve a bounded resampling risk without confidence statements on the p-values.

\subsection{A single decision requires infinite expected runtime}
\label{sec:single}
Following an argument similar to \cite[Section~3.1]{Gandy2009}, computing the decision of a hypothesis with random p-value requires an infinite expected runtime. Let $m=1$ hypothesis $H_{01}$ be given having a random p-value $p_1$. Assume a sequential algorithm $A$ tests $H_{01}$ at some given threshold $\alpha$ by approximating $p_1$ through the drawing of Monte Carlo simulations and gives a guarantee of $1-\epsilon$ on the correctness of its decision, where $\epsilon \in (0,1)$.

Computing a decision on $H_{01}$ is equivalent to deciding whether $p_1 \leq \alpha$ or $p_1 > \alpha$. For some $p_0>\alpha$, consider testing $H_0: p_1 \leq \alpha$ against $H_1: p_1=p_0$. A test can be constructed by rejecting $H_0$ if and only if $A$ does not reject $H_{01}$. Due to the guarantee of algorithm $A$, both the type 1 and type 2 errors of this test are $\epsilon$. For such a sequential test, a lower bound on the runtime $\tau$ (the expected number of steps) is given by \cite[equation (4.81)]{Wald1945} as
\begin{align}
	\E(\tau|p=p_1) \geq \frac{\epsilon \log \left( \frac{\epsilon}{1-\epsilon} \right) + (1-\epsilon) \log \left( \frac{1-\epsilon}{\epsilon} \right)} {p_1 \log \left( \frac{p_1}{\alpha} \right) + (1-p_1) \log \left( \frac{1-p_1}{1-\alpha} \right)}.
	\label{eq:wald}
\end{align}
The same bound \eqref{eq:wald} can be derived for the case $p_0 \leq \alpha$. Abbreviate the numerator of \eqref{eq:wald} by $C$ and consider $p_1$ in a Bayesian setup such that the following condition is satisfied.
\begin{condition}
	\label{condition:p}
	Assume $p_1$ has a distribution function $F(p_1)$ with derivative $F'(\alpha)>0$, and that for a suitable $\gamma>0$, there exists a constant $d>0$ such that $F'(p_1) \geq d$ in $(\alpha,\alpha+\gamma)$.
\end{condition}
Amongst others, Condition~\ref{condition:p} is satisfied for the distributions of the exponential family. Under Condition~\ref{condition:p}, $\E(\tau)$ can be bounded by
\begin{align*}
	\E(\tau) &= \int_0^1 \E(\tau|p=p_1) dF(p_1) \geq \int_{\alpha}^{\alpha+\gamma} \E(\tau|p=p_1) dF(p_1)\\
	&\geq C \cdot d \cdot \int_{\alpha}^{\alpha+\gamma} \left( p \log \left( \frac{p}{\alpha} \right) + (1-p) \log \left( \frac{1-p}{1-\alpha} \right) \right)^{-1} dp = \infty,
\end{align*}
as the integrand is proportional to $(p-\alpha)^{-2}$ as $p \rightarrow \alpha$. This proves an infinite expected runtime for the sequential test of $H_0$ and, equivalently, for algorithm $A$.

\section{Bonferroni-type multiple testing in expected finite time}
\label{sec:finite}
Assume the testing of $H_{01},\ldots,H_{0m}$ is carried out by comparing each $p_i$ to a threshold value $\alpha_i \in (0,1)$, $i \in \{ 1,\ldots,m \}$, as done in, for instance, step-up and step-down procedures \citep{GandyHahn2016}. For this assume $\alpha_1 < \cdots < \alpha_m$. In applications which rely on multiple testing at a constant (Bonferroni-type) threshold with a guarantee of correctness through confidence statements on the p-values, it will be shown that decisions on all but two hypotheses can be computed in expected finite time. The guarantee of correctness on all decisions is assumed to hold simultaneously for all hypotheses at $1-\epsilon$ for some pre-specified $\epsilon \in (0,1)$.

To be precise, a stronger statement is proven. For all but two hypotheses, it can be decided in expected finite time which of the intervals
\begin{align}
	{\cal I} = \{ [\alpha_i,\alpha_{i+1}): \alpha_i<\alpha_{i+1}, i=1,\ldots,m-1\} \cup \{ [0,\alpha_1), [\alpha_m,1] \}
	\label{eq:I}
\end{align}
their p-values fall into.

As p-values are unknown, they are approximated through Monte Carlo simulations, and confidence statements are provided via confidence intervals. Let $g(n)$ be the length of the confidence interval for a p-value after drawing $n$ Monte Carlo simulations, and $\hat{p}_n$ be the maximum likelihood estimate of $p$ based on $n$ Monte Carlo simulations. Alternatively, any other estimate $\hat{p}_n$ is permissible whose deviation from the mean can be bounded with a Hoeffding type inequality \citep{Hoeffding1963} (see the proof of Theorem~1 in the Supplementary Material).
\begin{condition}
	\label{condition:ci}
	Any confidence interval for $p$ contains $\hat{p}_n$. Moreover, $g(n) \in o(n^\gamma)$ for some $-\frac{1}{2} < \gamma < -\frac{1}{3}$.
\end{condition}
Define $D = \min_{I \in {\cal I}} d(p,\partial I)$ for a (random) $p$, where $d(p,\partial I)$ is the distance of $p$ to the boundary of $I$. If the distance of $p$ to $\hat{p}_n$ is less than $D/2$ and the confidence interval for $p$ has length $g(n) \leq D/2$, the confidence interval for $p$ will be entirely contained in some $I \in {\cal I}$ (assuming it contains $\hat{p}_n$ required by Condition~\ref{condition:ci}). Therefore, a decision on which interval $I \in {\cal I}$ contains $p$ is obtained on reaching the stopping time $\tau = \inf \{ n: |\hat{p}_n-p|<D/2, g(n)<D/2 \}$.

Let $\tau_1,\ldots,\tau_m$ be the stopping times of the $m$ p-values and $\tau_{(1)} \leq \cdots \leq \tau_{(m)}$ be their order statistic.

\begin{theorem}
	\label{thm:finite_expectation}
	Let the density of $p$ be bounded above by some finite constant. Assume $m \geq 3$. Under Condition~\ref{condition:ci}, $\E(\tau_{(m-s)}) < \infty$ for $2 \leq s < m$.
\end{theorem}

The proof of Theorem~\ref{thm:finite_expectation} is found in the Supplementary Material. For the \cite{Bonferroni1936} correction, determining a $I \in {\cal I}$ containing the confidence interval of each hypothesis implies determining if the p-value of a hypothesis is above or below the constant testing threshold (subject to the overall $1-\epsilon$ error probability), thus giving a decision on all but two hypotheses in finite expected time by Theorem~\ref{thm:finite_expectation}. This result does not extend to multiple testing applications which depend on all individual decisions, for instance step-up or step-down procedures, in which case no algorithm can guarantee even a single decision in finite expected time (Section~\ref{sec:multiple}).

Section~1 of the Supplementary Material contains two results which highlight the asymptotic length of popular confidence intervals. These results are now used to establish that Condition~\ref{condition:ci} is satisfied for a variety of confidence sequences and the algorithms they use:
\begin{enumerate}
	\item The length of \cite{ClopperPearson1934} confidence intervals is $g(n) \propto n^{-1/2} (-\log \rho_n)^{1/2}$ (see the proof in the Supplementary Material), where $\rho_n$ is a sequence controlling how the overall risk $\epsilon$ is spent. If $-\log(\rho_n) \propto \log(n)$ as in \cite{Gandy2009,GandyHahn2014}, then $g(n) \in o(n^\gamma)$ for any $-\frac{1}{2} < \gamma < -\frac{1}{3}$. Since the \cite{ClopperPearson1934} intervals contain $\hat{p}_n$, Condition~\ref{condition:ci} is satisfied. The intervals are employed in the algorithms of \cite{GuoPedadda2008, GandyHahn2014}.
	\item Confidence intervals produced by the binomial confidence sequences of \cite{Robbins1970, Lai1976} satisfy $g(n) \propto n^{-1/2} \{ \log(n \log n) \}^{1/2}$ (see the proof in the Supplementary Material). In \cite[Section~3(A)]{Lai1976} it is shown that the intervals contain $\hat{p}_n$, thus satisfying Condition~\ref{condition:ci}. These confidence sequences are employed in the methods of \cite{GandyHahn2016, Ding2018}.
	\item Intervals produced by the confidence sequences of \cite{DarlingRobbins1967b, DarlingRobbins1967a} have length $g(n) \propto n^{-1/2} \{ \log (\log n) \}^{1/2}$ \cite[Section~1]{DarlingRobbins1967b} and are centered around the empirical mean, thus satisfying Condition~\ref{condition:ci}.
	\item Binomial confidence intervals of \cite{Armitage1958} are employed in the methods of \cite{Fay2007,Gandy2009}. Being binomial exact intervals as the ones of \cite{ClopperPearson1934}, they satisfy Condition~\ref{condition:ci}.
	\item Condition~\ref{condition:ci} is satisfied for (asymptotic) maximum likelihood confidence intervals, given those are computed with a normal approximation. To be precise, the normal approximation produces confidence intervals of length $g(n) \propto n^{-1/2}$, thus satisfying Condition~\ref{condition:ci}.
	\item Bootstrap confidence intervals for $p$ can be written as $[\hat{p}+q_{\alpha},\hat{p}+q_{1-\alpha}]$ with $\alpha \in (0,1)$, where $q_{\alpha}$ and $q_{1-\alpha}$ are the empirical $\alpha$ and $1-\alpha$ quantiles of some bootstrap distribution $F^\ast$ approximating an underlying distribution $F$ that $p$ is drawn from. Although specific to the application under consideration, the results of Theorem~\ref{thm:finite_expectation} apply given it has been verified that $q_{1-\alpha}-q_{\alpha} \in o(n^\gamma)$ for some $-\frac{1}{2} < \gamma < -\frac{1}{3}$. In particular, this is true if a normal approximation is used as in the case of maximum likelihood confidence intervals.
\end{enumerate}

\section{Extension to infinite expected runtime for multiple testing}
\label{sec:multiple}
For multiple testing applications having the property that with non-zero probability, no decision on any hypothesis can be made until it is known whether a certain hypothesis is rejected or non-rejected (as it is the case, for instance, for step-up and step-down procedures), under reasonable assumptions on the distribution of p-values, an algorithm with both a guarantee of correctness and a finite expected runtime for any number of decisions cannot exist.

Consider testing the $m$ hypotheses $H_{01},\ldots,H_{0m}$ satisfying the following condition.
\begin{condition}
	\label{condition:multiple}
	The p-values $p=(p_1,\ldots,p_m)$ corresponding to $H_{01},\ldots,H_{0m}$ have a joint distribution with support $[0,1]^m$ and multivariate density $f_p(p)$. For all $\delta>0$ there exists a $\kappa>0$ such that $f_p(p)>\kappa$ on $[\delta,1-\delta]^m$.
\end{condition}
This condition is satisfied for many common densities. Assume testing is carried out with a step-up procedure that compares each $p_i$ to a threshold value $\alpha_i$, $i \in \{1,\ldots,m\}$, where $\alpha_1 < \cdots < \alpha_m$ (cf.\ Section~\ref{sec:finite}). For any $0<\eta<(\alpha_m-\alpha_{m-1})/4$, define $A = \left[ \alpha_{m-1} + 2\eta, \alpha_m - 2\eta \right]^{m-1} \times [\alpha_m-\eta,\alpha_m+\eta]$. Let $B=[\delta_B,1-\delta_B]^m$ and choose $\delta_B \in (0,1)$ in such a way that $A \cap B \neq \emptyset$.

Assuming the distribution of $p=(p_1,\ldots,p_m)$ satisfies Condition~\ref{condition:multiple} with $\delta=\delta_B$, draw a vector $\tilde{p}=(\tilde{p}_1,\ldots,\tilde{p}_m)$ from $f_p$ conditional on being in $A$. By definition of $A$, $\tilde{p}_m$ is the largest value in $\tilde{p}$ and will thus be compared to the threshold value $\alpha_m$. By properties of a step-up procedure, as $\alpha_{m-1} < \tilde{p}_i < \tilde{p}_m$ for all $i \neq m$, no decision on any hypothesis can be made unless it is known whether $\tilde{p}_m$ lies below or above $\alpha_m$. In the former case, all hypotheses are rejected, in the latter case, all hypotheses are non-rejected. Under the additional assumption that the marginal distribution of $p_m$ satisfies Condition~\ref{condition:p}, by Section~\ref{sec:single}, deciding $H_{0m}$ requires an infinite expected time. Thus on the event $\{p \in A\}$, the expected time to decide any number of hypotheses is also infinite, $\E(\tau|p \in A) = \infty$, where $\tau$ denotes the number of Monte Carlo simulations.

By the law of total expectation, the unconditional expected time can be bounded as $\E(\tau) \geq \E(\tau|p \in A) \cdot \PP(p \in A) = \infty$, where it was used that $\E(\tau|p \in A) = \infty$ and that $f_p(p)>\kappa>0$ on $B$ implies $\PP(p \in A) \geq \PP(p \in A \cap B) > 0$.

The above consideration proves an infinite expected runtime for step-up procedures for any desired number of decisions. It holds true, for instance, for the step-up procedures of \cite{Simes1986,Hochberg1988,Rom1990,BenjaminiHochberg1995,BenjaminiYekutieli2001}. A similar construction proves the same result for step-down procedures \citep{Sidak1967,Holm1979,Shaffer1986}. Extensions to other testing applications in which obtaining any decision can be made dependent on the decision of a single hypothesis are possible but application-specific.

\section{Conclusions}
\label{sec:conclusions}
\begin{figure}
	\centering
	\includegraphics[width=0.5\textwidth]{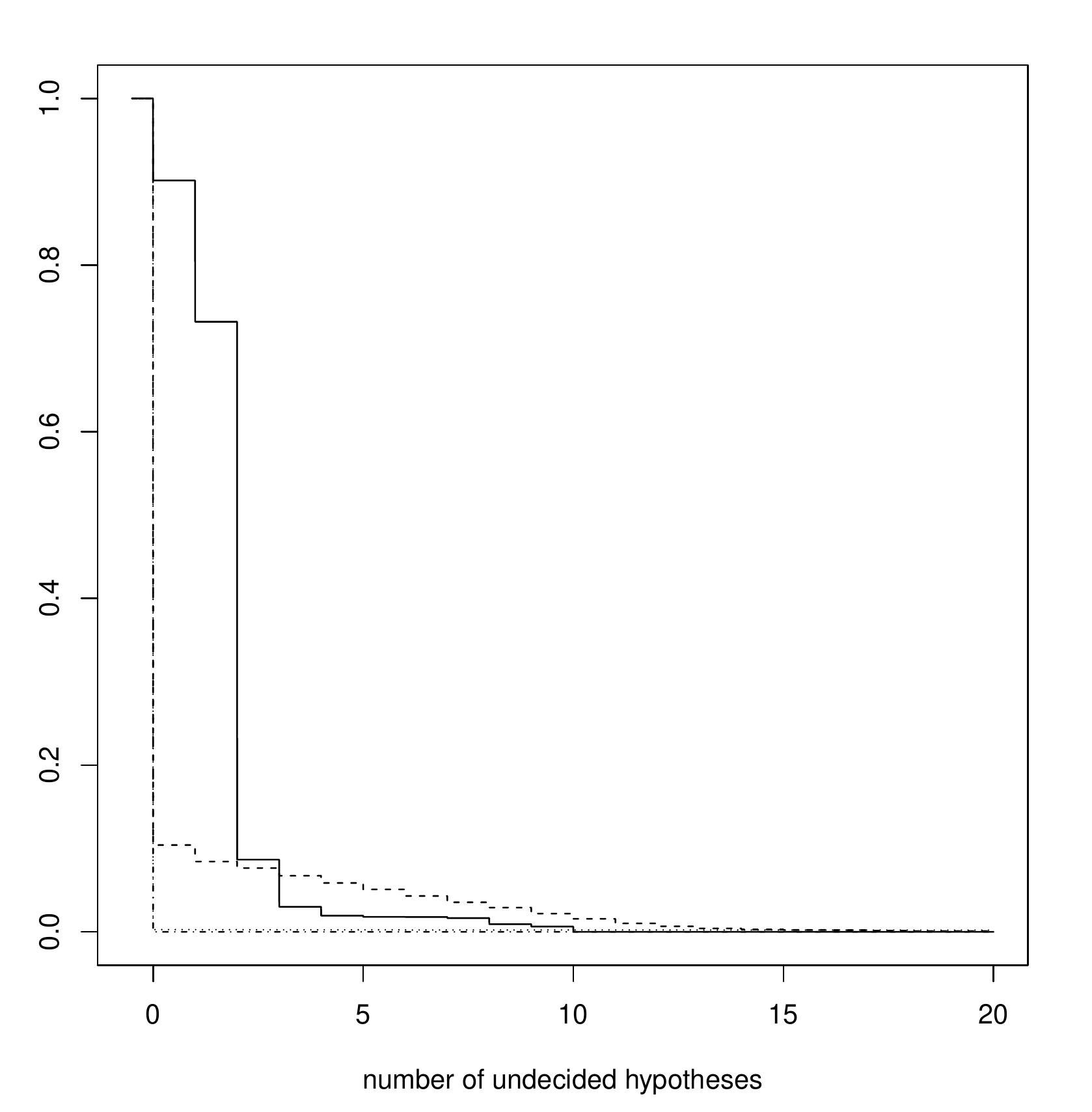}
	\caption{Survival function of the number of undecided hypotheses tested with the \cite{BenjaminiHochberg1995} procedure with threshold $\alpha=0.1$. $R=10^4$ repetitions. Number of hypotheses $m$ from $10^1$ (solid), $10^2$ (dashed), $10^3$ (dotted), $10^4$ (dash-dotted).\label{fig:survival}}
\end{figure}
According to Theorem~\ref{thm:finite_expectation}, it can be decided in expected finite time for all but two hypotheses which interval $I \in {\cal I}$ in \eqref{eq:I} their p-values are contained in. For independent (Bonferroni-type) testing, this implies at most two pending decisions in finite expected time.

For step-up and step-down procedures, the situation is more involved. This is because the individual decision on a hypothesis need not matter in step-up (step-down) procedures so long as there exists another hypothesis with a larger (smaller) p-value for which a decision is available. As a consequence, having two undecided hypotheses in the sense of Theorem~\ref{thm:finite_expectation} is not always representative of the actual number of decisions. Although in Section~\ref{sec:multiple}, this fact was used to show that under conditions, the expected time to compute any decision is infinite, in practice the decisions on a large number of hypotheses can often be computed quickly. For the \cite{BenjaminiHochberg1995} procedure with threshold $\alpha=0.1$, Fig.~\ref{fig:survival} displays the survival function of the number of undecided hypotheses which remain if, out of $m \in \{ 10^1,10^2,10^3,10^4 \}$ hypotheses, for only $m-2$ hypotheses it can be decided which $I \in {\cal I}$ in \eqref{eq:I} their p-values are contained in. The figure is based on $R=1000$ repetitions using p-values generated from the mixture distribution of \cite{Sandve2011}, consisting of a proportion $\pi_0=0.8$ drawn from a uniform distribution in $[0,1]$ and the remaining proportion $1-\pi_0$ drawn from a beta$(0.5,25)$ distribution. As can be seen, often only a handful of hypotheses remain without decision. On this note, a related question addresses the optimal strategy (where optimality is defined with respect to the expected number of erroneously classified hypotheses) for allocating Monte Carlo simulations to multiple hypotheses in order to maximize the accuracy of the testing result. This has been addressed in the literature \citep{HahnOptimal}.

\appendix
\section{Appendix}
In the following, two versions of Hoeffding's inequality \citep{Hoeffding1963} are used in several places. Specifically, let $X_1,\ldots,X_n$ be independent random variables which are almost surely bounded, that is $\PP(X_i \in [a_i,b_i])=1$ for some constants $a_i, b_i$ and all $i \in \{1,\ldots,n\}$. Then by \cite[Theorem~2]{Hoeffding1963} the empirical mean $\bar{X}=(X_1+\cdots+X_n)/n$ satisfies
\begin{align*}
	\PP \left( \bar{X}-\E(\bar{X}) \leq \delta \right) &\leq \exp \left( - \frac{2 \delta^2 n^2}{\sum_{i=1}^n (b_i-a_i)^2} \right),\\
	\PP \left( \left| \bar{X}-\E(\bar{X}) \right| \geq \delta \right) &\leq 2 \exp \left( - \frac{2 \delta^2 n^2}{\sum_{i=1}^n (b_i-a_i)^2} \right).
\end{align*}
In all following proofs, both inequalities will always be applied to a sum of independent Bernoulli random variables, in which case $b_i=1$ and $a_i=0$ for all $i \in \{1,\ldots,n\}$, and the denominator of the exponentials above can be simplified to $\sum_{i=1}^n (b_i-a_i)^2=n$.

\subsection{Auxiliary lemmas}
\label{sec:lemmas}

\begin{lemma}
\label{lemma:cp}
The two-sided \cite{ClopperPearson1934} confidence interval $I_n$ with coverage probability $1-\rho_n$ based on $n$ Monte Carlo simulations has length $|I_n| \leq 2 (2n)^{-1/2} (-\log \rho_n)^{1/2}$.
\end{lemma}
\begin{proof}
Suppose $S<n$ exceedances are observed among $n$ Monte Carlo simulations. Let $\xi=(2n)^{-1/2} (-\log \rho_n)^{1/2}$ and regard the following probabilities conditional on $S$ and $n$. The upper limit $p_u$ of the interval $I_n$ is the solution to $\PP(X \leq S|p=p_u) = \rho_n$, where  $X \sim \text{Binomial}(n,p)$. If $p > S/n +\xi$, by Hoeffding's inequality,
$$\PP(X \leq S)
= \PP \left( \frac{X}{n} - \E \left( \frac{X}{n} \right)
\leq \frac{S}{n} - \E \left( \frac{X}{n} \right) \right)
\leq \exp \left( - \frac{2 (S/n-p)^2 n^2}{n} \right)
< \rho_n,$$
where the first equality was obtained by dividing by $n$ and subtracting $\E(X/n)=p$ on both sides inside the probability, and where the Binomial variable $X$ was expressed as $X=X_1+\cdots+X_n$ for some Bernoulli variables $X_i$, $i \in \{1,\ldots,n\}$. Thus $p_u \leq S/n + \xi$. If $S=n$ then $p_u=1$, implying $p_u \leq S/n + \xi$. Similarly, the lower limit $p_l$ of $I_n$ satisfies $p_l \geq S/n - \xi$. Together, $|I_n| = p_u-p_l \leq 2 \xi$.
\end{proof}

The proof of Lemma~\ref{lemma:cp} is analogous to the one of \cite[Lemma~2]{GandyHahn2014}.

\begin{lemma}
\label{lemma:lai}
The confidence interval $I_n$ produced by the \cite{Robbins1970} and \cite{Lai1976} Binomial confidence sequence based on $n$ Monte Carlo simulations has length $|I_n| \leq n^{-1/2} \{ \log(4n \log n) \}^{1/2}$ for large enough $n$.
\end{lemma}
\begin{proof}
In \cite[equation~(3)]{Lai1976} it is shown that $\PP(p \notin I_n) \propto (n \log n)^{-1/2}$. Since $I_n$ need not be centered around the maximum likelihood estimate $\hat{p}_n$, let $I_n^s \supseteq I_n$ be the smallest symmetric interval around $\hat{p}_n$ containing $I_n$. Denote $I_n^s = [\hat{p}_n-t_n,\hat{p}_n+t_n]$ for some $t_n>0$. Write $\hat{p}_n$ as the average of $n$ independent Bernoulli($p$) random variables. By Hoeffding's inequality, using $\E(\hat{p}_n)=p$, one obtains $\PP(p \notin I_n^s) = \PP(|\hat{p}_n-p|>t_n) \leq 2 \exp(-2 n t_n^2)$. Since $\PP(p \notin I_n^s) \leq \PP(p \notin I_n)$, $\PP(p \notin I_n^s)$ is at least of order $(n \log n)^{-1/2}$. Thus, $t_n \leq (4n)^{-1/2} \{ \log(4n \log n) \}^{1/2}$ for large enough $n$.
\end{proof}

\subsection{Proof of Theorem~\ref{thm:finite_expectation}}
\label{sec:proof}
\begin{proof}
The cdf of $D = \min_{I \in {\cal I}} d(p,\partial I)$ is bounded above by
\begin{align*}
\PP(D \leq t) = \PP(\exists I \in {\cal I}: d(p,\partial I) \leq t) \leq \sum_{I \in {\cal I}} \int_0^t d(p,\partial I) f_p(p) dp \leq (m+1) \int_0^t U dp \leq (m+1)U t,
\end{align*}
where it was used that the cardinality of $\cal I$ is at most $m+1$, that the distance function $d$ is bounded above by $1$, and that the density $f_p$ of the p-values is bounded above by a constant $U$.

In the following, an upper bound on the survival function $\PP(\tau>t)$ will be derived. First, $\PP(|\hat{p}_t-p| \geq D/2) = \PP(|\hat{p}_t-\E(\hat{p}_t)| \geq D/2) \leq 2 \exp(-D^2 t/2)$ by Hoeffding's inequality, where it was used that the maximum likelihood estimate $\hat{p}_t$ can be expressed as an average of $t$ Bernoulli($p$) random variables. Second, using $g(t) \in o(t^\gamma)$ by Condition~\ref{condition:ci}, there exists a $t_0>0$ such that the event $\{ g(t)<D/2 \}$ is implied by $\{ D>t^\gamma \}$ for all $t>t_0$. Indeed, $g(t) \in o(t^\gamma)$ and $D t^{-\gamma}>1$ imply the existence of a $t_0>0$ such that $g(t) < g(t) D t^{-\gamma} < D/2$ for all $t>t_0$.

The survival function $\PP(\tau>t)$ can now be bounded above by conditioning on $D$. For $t>t_0$,
\begin{align*}
\PP(\tau > t) &= \PP(\tau > t|D \leq t^\gamma) \PP(D \leq t^\gamma)
+ \PP(\tau > t|D > t^\gamma) \PP(D > t^\gamma)\\
&\leq (m+1)U t^\gamma + \PP(|\hat{p}_t-p| \geq D/2 \vee g(t) \geq D/2 |D > t^\gamma)\\
&\leq (m+1)U t^\gamma + \PP(|\hat{p}_t-p| \geq D/2 | D > t^\gamma)\\
&\leq (m+1)U t^\gamma + 2 \exp \left( -\frac{1}{2} t^{2\gamma+1} \right),
\end{align*}
where it was used that $\PP(\tau > t|D \leq t^\gamma) \leq 1$, $\PP(D > t^\gamma) \leq 1$ and that $\tau>t$ if either $|\hat{p}_t-p| \geq D/2$ or $g(t) \geq D/2$ by the definition of $\tau$ in Section~\ref{sec:finite}. The latter can be omitted as $\{ D>t^\gamma \}$ implies $\{ g(t)<D/2 \}$ for $t>t_0$. Using $\gamma > -1/2$ by Condition~\ref{condition:ci}, $t^{2\gamma+1} \rightarrow \infty$ in the argument of the exponential function as $t \rightarrow \infty$ and thus $\PP(\tau > t) \in O(t^\gamma)$.

Using the fact that the cumulative distribution function (cdf) of the $r$th order statistic $X_{(r)}$ of $n \in \N$ independent and identically distributed random variables $X_1,\ldots,X_n$ with cdf $F_X$ can be expressed as $F_{X_{(r)}}(x) = \sum_{i=r}^{n} \binom{n}{i} F_X^i(x) \left( 1-F_X(x) \right)^{n-i}$ for $r \in \{ 1,\ldots,n \}$ \citep{DavidNagaraja2003}, the expectation of $\tau_{(m-s)}$ can be bounded as
\begin{align*}
\E(\tau_{(m-s)}) &= \int_0^\infty 1-F_{\tau_{(m-s)}}(t) dt\\
&= \int_0^\infty \sum_{i=0}^{m-s-1} \binom{m}{i} F_\tau^i(t) \left( 1-F_\tau(t) \right)^{m-i} dt
\leq \sum_{i=0}^{m-3} \binom{m}{i} \int_0^\infty \PP(\tau>t)^{m-i} dt,
\end{align*}
where it was used that the stopping times are non-negative, that $F_\tau(t) \leq 1$ and that $2 \leq s < m$ implies $m-s-1 \leq m-3$. Using $\PP(\tau > t) \in O(t^\gamma)$, the integrals $\int_0^\infty \PP(\tau>t)^{m-i} dt$ behave like $\int_0^\infty t^{\gamma(m-i)} dt$ and converge as $\gamma < -1/3$ (cf.\ Condition~\ref{condition:ci}) and $m \geq 3$ imply $\gamma(m-i) < -1$ for all $i \in \{0,\ldots,m-3 \}$.
\end{proof}


\begin{thebibliography}{}
\bibitem[Andrews and Buchinsky, 2000]{AndrewsBuchinsky2000}
Andrews, D. and Buchinsky, M. (2000).
\newblock {A Three‐step Method for Choosing the Number of Bootstrap
  Repetitions}.
\newblock {\em Econometrica}, 68(1):23--51.

\bibitem[Andrews and Buchinsky, 2003]{AndrewsBuchinsky2003}
Andrews, D. and Buchinsky, M. (2003).
\newblock Evaluation of a three-step method for choosing the number of
  bootstrap repetitions.
\newblock {\em J Econometrics}, 103(1-2):345--386.

\bibitem[Armitage, 1958]{Armitage1958}
Armitage, P. (1958).
\newblock {Numerical Studies in the Sequential Estimation of a Binomial
  Parameter}.
\newblock {\em Biometrika}, 45(1-2):1--15.

\bibitem[Benjamini and Hochberg, 1995]{BenjaminiHochberg1995}
Benjamini, Y. and Hochberg, Y. (1995).
\newblock {Controlling the false discovery rate: A practical and powerful
  approach to multiple testing}.
\newblock {\em J Roy Stat Soc B Met}, 57(1):289--300.

\bibitem[Benjamini and Yekutieli, 2001]{BenjaminiYekutieli2001}
Benjamini, Y. and Yekutieli, D. (2001).
\newblock The control of the false discovery rate in multiple testing under
  dependency.
\newblock {\em Ann Stat}, 29(4):1165--1188.

\bibitem[Besag and Clifford, 1991]{BesagClifford1991}
Besag, J. and Clifford, P. (1991).
\newblock {Sequential Monte Carlo p-values}.
\newblock {\em Biometrika}, 78(2):301--304.

\bibitem[Bonferroni, 1936]{Bonferroni1936}
Bonferroni, C. (1936).
\newblock Teoria statistica delle classi e calcolo delle probabilit\`a.
\newblock {\em Pubblicazioni del R Istituto Superiore di Scienze Economiche e
  Commerciali di Firenze}, 8:3--62.

\bibitem[Clopper and Pearson, 1934]{ClopperPearson1934}
Clopper, C. and Pearson, E. (1934).
\newblock {The Use of Confidence or Fiducial Limits Illustrated in the Case of
  the Binomial}.
\newblock {\em Biometrika}, 26(4):404--413.

\bibitem[Darling and Robbins, 1967a]{DarlingRobbins1967b}
Darling, D. and Robbins, H. (1967a).
\newblock Confidence sequences for mean, variance, and median.
\newblock {\em P Natl Acad Sci USA}, 58(1):66--68.

\bibitem[Darling and Robbins, 1967b]{DarlingRobbins1967a}
Darling, D. and Robbins, H. (1967b).
\newblock Iterated logarithm inequalities.
\newblock {\em P Natl Acad Sci USA}, 57(5):1188--92.

\bibitem[David and Nagaraja, 2003]{DavidNagaraja2003}
David, N. and Nagaraja, H. (2003).
\newblock {\em {Order Statistics}}.
\newblock Wiley.

\bibitem[Davidson and MacKinnon, 2000]{DavidsonMacKinnon2000}
Davidson, R. and MacKinnon, J. (2000).
\newblock {Bootstrap tests: How many bootstraps?}
\newblock {\em Economet Rev}, 19(1):55--68.

\bibitem[Ding et~al., 2018]{Ding2018}
Ding, D., Gandy, A., and Hahn, G. (2018).
\newblock {A simple method for implementing Monte Carlo tests}.
\newblock {\em arXiv:1611.01675}, pages 1--17.

\bibitem[Fay and Follmann, 2002]{FayFollmann2002}
Fay, M. and Follmann, D. (2002).
\newblock Designing {M}onte {C}arlo implementations of permutation or bootstrap
  hypothesis tests.
\newblock {\em Am Stat}, 56(1):63--70.

\bibitem[Fay et~al., 2007]{Fay2007}
Fay, M., Kim, H.-J., and Hachey, M. (2007).
\newblock On using truncated sequential probability ratio test boundaries for
  {M}onte {C}arlo implementation of hypothesis tests.
\newblock {\em J Comput Graph Stat}, 16(4):946--967.

\bibitem[Gandy, 2009]{Gandy2009}
Gandy, A. (2009).
\newblock {Sequential Implementation of Monte Carlo Tests With Uniformly
  Bounded Resampling Risk}.
\newblock {\em J Am Stat Assoc}, 104(488):1504--1511.

\bibitem[Gandy and Hahn, 2014]{GandyHahn2014}
Gandy, A. and Hahn, G. (2014).
\newblock {MMCTest -- A Safe Algorithm for Implementing Multiple Monte Carlo
  Tests}.
\newblock {\em Scand J Stat}, 41(4):1083--1101.

\bibitem[Gandy and Hahn, 2016]{GandyHahn2016}
Gandy, A. and Hahn, G. (2016).
\newblock {A Framework for Monte Carlo based Multiple Testing}.
\newblock {\em Scand J Stat}, 43(4):1046--1063.

\bibitem[Gandy and Hahn, 2017]{GandyHahn2017}
Gandy, A. and Hahn, G. (2017).
\newblock {QuickMMCTest: quick multiple Monte Carlo testing}.
\newblock {\em Stat Comput}, 27(3):823--832.

\bibitem[Guo and Peddada, 2008]{GuoPedadda2008}
Guo, W. and Peddada, S. (2008).
\newblock {Adaptive Choice of the Number of Bootstrap Samples in Large Scale
  Multiple Testing}.
\newblock {\em Stat Appl Genet Mol Biol}, 7(1):1--16.

\bibitem[Hahn, 2019]{HahnOptimal}
Hahn, G. (2019).
\newblock {Optimal allocation of Monte Carlo simulations to multiple hypothesis
  tests}.
\newblock {\em Stat Comput}, pages 1--16.

\bibitem[Hochberg, 1988]{Hochberg1988}
Hochberg, Y. (1988).
\newblock {A sharper Bonferroni procedure for multiple tests of significance}.
\newblock {\em Biometrika}, 75(4):800--802.

\bibitem[Hoeffding, 1963]{Hoeffding1963}
Hoeffding, W. (1963).
\newblock Probability inequalities for sums of bounded random variables.
\newblock {\em J Am Stat Assoc}, 58(301):13--30.

\bibitem[Holm, 1979]{Holm1979}
Holm, S. (1979).
\newblock {A Simple Sequentially Rejective Multiple Test Procedure}.
\newblock {\em Scand J Stat}, 6(2):65--70.

\bibitem[Kim, 2010]{Kim2010}
Kim, H.-J. (2010).
\newblock Bounding the resampling risk for sequential {M}onte {C}arlo
  implementation of hypothesis tests.
\newblock {\em J Stat Plan Infer}, 140(7):1834--1843.

\bibitem[Lai, 1976]{Lai1976}
Lai, T. (1976).
\newblock {On Confidence Sequences}.
\newblock {\em Ann Stat}, 4(2):265--280.

\bibitem[Lin, 2005]{Lin2005}
Lin, D. (2005).
\newblock An efficient {M}onte {C}arlo approach to assessing statistical
  significance in genomic studies.
\newblock {\em Bioinformatics}, 21(6):781--787.

\bibitem[Robbins, 1970]{Robbins1970}
Robbins, H. (1970).
\newblock Statistical methods related to the law of the iterated logarithm.
\newblock {\em Ann Math Statist}, 41(5):1397--1409.

\bibitem[Rom, 1990]{Rom1990}
Rom, D. (1990).
\newblock A sequentially rejective test procedure based on a modified
  {B}onferroni inequality.
\newblock {\em Biometrika}, 77(3):663--665.

\bibitem[Sandve et~al., 2011]{Sandve2011}
Sandve, G., Ferkingstad, E., and Nyg\r{a}rd, S. (2011).
\newblock {Sequential Monte Carlo multiple testing}.
\newblock {\em Bioinformatics}, 27(23):3235--3241.

\bibitem[Shaffer, 1986]{Shaffer1986}
Shaffer, J. (1986).
\newblock {Modified Sequentially Rejective Multiple Test Procedures}.
\newblock {\em J Am Stat Assoc}, 81(395):826--831.

\bibitem[Sidak, 1967]{Sidak1967}
Sidak, Z. (1967).
\newblock Rectangular confidence regions for the means of multivariate normal
  distributions.
\newblock {\em J Am Stat Assoc}, 62(318):626--633.

\bibitem[Silva and Assun\c{c}\~ao, 2013]{SilvaAssuncao2013}
Silva, I. and Assun\c{c}\~ao, R. (2013).
\newblock Optimal generalized truncated sequential {M}onte {C}arlo test.
\newblock {\em J Multivariate Anal}, 121:33--49.

\bibitem[Silva and Assun\c{c}\~ao, 2018]{SilvaAssuncao2018}
Silva, I. and Assun\c{c}\~ao, R. (2018).
\newblock Truncated sequential {M}onte {C}arlo test with exact power.
\newblock {\em Braz J Probab Stat}, 32(2):215--238.

\bibitem[Silva et~al., 2009]{Silva2009}
Silva, I., Assun\c{c}\~ao, R., and Costa, M. (2009).
\newblock Power of the sequential {M}onte {C}arlo test.
\newblock {\em Sequential Anal}, 28(2):163--174.

\bibitem[Simes, 1986]{Simes1986}
Simes, R. (1986).
\newblock {An improved Bonferroni procedure for multiple tests of
  significance}.
\newblock {\em Biometrika}, 73(3):751--754.

\bibitem[van Wieringen et~al., 2008]{Wieringen2008}
van Wieringen, W., van~de Wiel, M., and van~der Vaart, A. (2008).
\newblock {A Test for Partial Differential Expression}.
\newblock {\em J Am Stat Assoc}, 103(483):1039--1049.

\bibitem[Wald, 1945]{Wald1945}
Wald, A. (1945).
\newblock {Sequential Tests of Statistical Hypotheses}.
\newblock {\em Ann Math Statist}, 16(2):117--186.
\end{thebibliography}

\end{document}